\documentclass[a4paper,10pt]{article}

\usepackage{amsmath}
\usepackage{amssymb}
\usepackage{enumerate}
\usepackage{algorithm}
\usepackage{algorithmic}
\usepackage{amsfonts}
\usepackage{amsthm}
\usepackage{bbm}
\usepackage{bm}
\usepackage[mathscr]{eucal}
\usepackage{a4wide}
\usepackage{authblk}
\usepackage{tikz}
\usetikzlibrary{decorations.pathreplacing}
\usepackage{accents}
\usepackage{color}
\usepackage{doi}
\usepackage[square,numbers]{natbib}
\usepackage{tikz}
\usetikzlibrary{decorations.pathreplacing}

\definecolor{pvdh_lightblue}{RGB}{91,155,213}
\definecolor{pvdh_darkblue}{RGB}{0,74,127}
\definecolor{pvdh_darkgreen}{RGB}{65,123,133}
\definecolor{pvdh_purple}{RGB}{134,34,181}

\usepackage[numbib]{tocbibind}
\usepackage{hyperref}
\hypersetup{
    colorlinks = true,
    urlcolor = {pvdh_purple},
    citecolor = {pvdh_darkgreen},
    linkcolor = {pvdh_darkblue}
}

\newcommand{\Prob}[1]{\mathbb{P}\left(#1\right)}

\newcommand{\Exp}[1]{\mathbb{E}\left[#1\right]}
\newcommand{\Expc}[2]{\mathbb{E}_{#1}\left[#2\right]}
\newcommand{\CExp}[2]{\mathbb{E}\left[\left.#1\right|#2\right]}

\newcommand{\plim}{\ensuremath{\stackrel{\mathbb{P}}{\rightarrow}}}

\newcommand\numberthis{\addtocounter{equation}{1}\tag{\theequation}}

\newcommand*{\swap}[2]{\hspace{-0.5ex}#2#1}

\newtheorem{theorem}{Theorem}[section]
\newtheorem{definition}{Definition}[section]
\newtheorem{lemma}[theorem]{Lemma}
\newtheorem{proposition}[theorem]{Proposition}

\newtheorem{remark}[theorem]{Remark}

\title{Limiting Behavior of Degree-Degree Metrics under Local Convergence in Probability}
\author[1]{Andrei-Eugeniu Pătularu}
\author[2]{Pim van der Hoorn}

\affil[1]{École Polytechnique Fédérale de Lausanne}
\affil[2]{Department of Mathematics and Computer Science, Eindhoven University of Technology}

\begin{document}

\maketitle

\begin{abstract}
This paper investigates the limiting behaviour of degree-degree correlation metrics for sequences of random graphs under the general assumption of local convergence in probability. We establish convergence results for Pearson's correlation coefficient \(r\), Spearman’s \(\rho\), Kendall's \(\tau\), average nearest neighbour degree (ANND), and average nearest neighbour rank (ANNR). Our result explicitly show how the limits of these degree-degree correlation metrics depend on local structure of the graph. We then apply our general results to study degree-degree correlations in rank-1 inhomogeneous random graphs and random geometric graphs, deriving explicit expressions for ANND in both models and for Pearson's correlation coefficient in the latter one.

\end{abstract}

\tableofcontents

\section{Introduction}
In today's interconnected world, networks are everywhere, from social media connections to the intricate patterns of the internet. These networks can range from simple designs to immensely complex structures involving millions of interconnected nodes. To understand the function and behavior of these complex networks, researchers study their structural features and develop random graph models that mimic these. 

An important property that involves all networks is the degree-degree correlation, sometimes called network assortativity, which refers to the statistical relationship between the degrees of neighbouring nodes in a network. It quantifies how the degree of one node is related to the degrees of its adjacent nodes. For instance, if a network has a positive degree-degree correlation, it indicates that nodes of a high degree have a preference to connect to other high-degree nodes. In this case, the network is said to exhibit assortative mixing. Similarly, there exist disassortative networks, with negative degree-degree correlation, where nodes of high degree are mostly neighbours of nodes with small degrees. If the network is neither assortative nor disassortative, it is said to have neutral mixing. 

Degree-degree correlations are an important property of networks. For example, network with assortative mixing might be more vulnerable to targeted attacks where the high-degree nodes are specifically removed~\cite{newman2003Mixing}. In neuroscience assortative networks brain networks are shown to perform better in terms of signal processing~\cite{Schmeltzer2015degree}. Assortative networks are also more robust under edge or vertex removal. 
In financial networks, for example, assortativity may influence systemic risk, since highly interconnected banks are also highly connected to other similar entities \cite{Minoiu2013connectedness}.
Conversely, networks with disassortative mixing, where high-degree nodes are preferentially connected to low-degree nodes, can behave differently under stress or failure scenarios compared to assortative networks~\cite{newman2002assortative}. 
Disassortative networks do allow for easier immunization when considering epidemic spreading~\cite{DAgostino2012robustness}.

Given the impact of degree-degree correlations on the function of networks, it is important to properly measure and analyze these correlations. One natural way to do this is to study the asymptotic behaviour of degree-degree metrics in random graph models, as the number of nodes in the network grows large. There are many results available. Some concern specific models~\cite{nikoloski2005degree,antonioni2012degree,stegehuis2019degree,mann2022degree,kaufmann2025assortativity}, while others prove limits for a given metric under certain assumptions on the random graph model~\cite{hofstad2014degree,hoorn2016asymptotic,yao2018average}. While the latter have the potential to enable analysis of degree-degree correlation in general random graphs, the results are not always easy to directly apply, often resulting in a separate, and sometimes involved, analysis. 

The recent development of the local convergence has opened up a powerful and general framework to study sparse random graphs, enabling a uniform way to analyze topologies of a wide variety of random graph models. Specifically, if $(G_{n})_{n \ge 1}$ is a sequence of finite graphs, local convergence means that the distribution of neighbourhoods around a uniformly sampled node converges to the distribution of neighbourhoods in an infinitely rooted random graph. This notion is particularly useful since it implies convergence of local properties of random graphs, while a wide range of random graph models have been shown to have a local limit~\cite{hofstad2023local}. The notion of local convergence is by now the setting to analyze random graph models and has led to a wide variety of properties being studied. However, apart from some initial results, degree-degree metrics have not been extensively studied, even though most of them are local properties. 

In this paper we address this gap by providing general limit results for a wide range of degree-degree metrics for random graphs with a local limit. This makes our results widely applicable to most of random graph models that are currently studied, including the popular Geometric Inhomogeneous Random Graphs~\cite{bringmann2019geometric}, Weighted Random Connection Model~\cite{gracar2021percolation} and the more general Spatial Inhomogeneous Random Graphs~\cite{hofstad2023local}. To showcase the usage of our results we analyze degree-degree correlations in rank-1 inhomogeneous random graphs and random geometric graphs. We provide an explicit expression for the Average Nearest Neighbor Degree in both models and for Pearson's correlation coefficient in the latter one.

We organize the paper as follows. In Section~\ref{sec:main_results} we provide basic preliminaries for local convergence and state our main convergence results for the different degree-degree metrics. The applications of our general results to rank-1 inhomogeneous random graphs and random geometric graphs is covered in Section~\ref{sec:applications}. Sections~\ref{sec:main_proofs} and~\ref{sec:application_proofs} contain the proofs of the general results and the applications, under the assumption of a few technical lemmas. The proofs for these lemmas are included in the Appendix.

\section{Preliminaries and main results}\label{sec:main_results}

\subsection{Local convergence}

We start by briefly introducing the concept of local convergence for a sequence of graphs as introduced in~\cite{aldous2004objective,benjamini2011recurrence,hofstad2024random}. In particular we focus on the notion of local convergence in probability. This notion focuses on the behavior of local neighborhoods around a random vertex as the graph grows in size. The interested reader is referred to~\cite{hofstad2024random} for more details on the topic. 

For a graph $G = (V(G),E(G))$ we denote by $d_v$ the degree of $v \in V(G)$. The graph $G$ is called \emph{locally finite} if $d_v < \infty$ holds for all $v \in V(G)$. 

\begin{definition}[Rooted graphs]\label{def:rooted_graphs} A rooted graph is a tuple $(G, o)$, where $G = (V(G), E(G))$ is a graph with a designated vertex $o \in V(G)$ called the root.    
\end{definition}

For any (undirected) graph $G= (V(G), E(G))$ we denote by $d_{G}(u,v)$ the graph distance between nodes $u,v$ in $G$, i.e. the length of the shortest path between $u$ and $v$. We adopt the convention $d_{G}(u,v) = \infty$, if $u$ and $v$ are not connected in $G$.  

\begin{definition}[Neighborhood of the root] Let $(G,o)$ be a locally finite rooted graph. Then, for every $r>0$ we denote by $B_{r}^{(G)}(o)$ the subgraph of $G$ induced by  
\[
	\{v \in V(G): d_{G}(o,v) \le r\}.
\]

Informally, $(B_{r}^{(G)}(o), o)$ is the rooted subgraph of $(G,o)$ having all vertices at graph distance at most $r$ from the root $o$.
\end{definition}

Next, we introduce the notion of rooted isomorphism, which aligns with the usual definition of graph isomorphism.

\begin{definition}[Rooted isomorphisms]\label{def:rooted_isomorphism} Let $(G_{1}, o_{1})$ and $(G_{2}, o_{2})$ be two locally finite rooted graphs. Then $(G_{1}, o_{1})$ is rooted isomorphic to $(G_{2}, o_{2})$ when there exists a bijective function $\phi : V(G_{1}) \rightarrow V(G_{2})$ satisfying \[\{u,v\} \in E(G_{1}) \iff \{\phi(u), \phi(v)\} \in E(G_{2}) \]
and $\phi(o_{1}) =o_{2}$. Moreover, we use the notion $(G_{1}, o_{1}) \simeq (G_{2}, o_{2})$ when  $(G_{1}, o_{1})$ is rooted isomorphic to $(G_{2}, o_{2})$.
\end{definition}

Thus, by using Definition \ref{def:rooted_isomorphism} we let $\mathcal{G}_{\star}$ be the space of rooted graphs modulo the isomorphism, i.e. $\mathcal{G}_{\star}$ consists the set of all equivalence classes of the form $[(G,o)]$, where $(G',o') \in [(G,o)]$ if and only if $(G',o') \simeq (G,o)$. However, we often omit the equivalence class notation and adopt the convention $(G,o) \in \mathcal{G}_{\star}$, in particular, if $(G',o') \in [(G,o)]$, i.e. $(G',o') \simeq (G,o)$ then $(G,o)$ and $(G', o')$ are considered to be the same. This space can be turned into a Polish space with using the metric
\[
	d((G_1,o_1), (G_2,o_2)) = \frac{1}{1 + \sup_{r > 0} \{B_r^{G_1}(o_1) \simeq B_r^{G_2}(o_2)\}}.
\] 
(see~\cite{hofstad2024random}) and thus can be turned into a probability space.

We are now ready to state the definition of local convergence in probability. For a sequence $X, (X_n)_{n \ge 1}$ of random variables we write $X_n \plim X$ to denote convergence in probability of $X_n$ to $X$.

\begin{definition}{(Local convergence in probability)}\label{def:local_convergence_probability} Let $(G_{n})_{n \ge 1}$ be a sequence of (possibly random) graphs that are almost surely finite. Let $(G_{n}, o_{n})$ be the corresponding sequence of rooted graphs, where the root $o_{n} \in V(G_{n})$ is chosen uniformly at random and we restrict $G_n$ to the connected component of $o_n$. Then we say $(G_{n}, o_{n})$ \emph{converges locally in probability} to the (possibly random) connected rooted graph $(G, o) \in \mathcal{G}_{\star}$, having law $\mu$, if for every rooted graph $H_{\star} \in \mathcal{G}_{\star}$ and all integers $r \ge 0$ it holds that
\begin{equation}
    \frac{1}{|V(G_{n})|}\sum_{v \in V(G_{n})}\mathbbm{1}_{\{B_{r}^{(G_{n})}(v) \simeq H_{\star}\}} \xrightarrow{\mathbb{P}} \mu(B_{r}^{(G)}(o) \simeq H_{\star})~ \text{as $n \to \infty$}.
\end{equation}
If $(G_{n}, o_{n})$ converges locally in probability to $(G,o)$ then, slightly abusing notation, we write $(G_{n},o_{n}) \xrightarrow{\mathbb{P}} (G,o)$. 
\end{definition}

\subsection{Statement of main results}

In this section, provide convergence results for several different metrics for degree-degree correlations. To analyze these correlations it will be useful to consider directed edges in an undirected graph. Thus, for any (deterministic) undirected graph $G= (V(G), E(G))$  we consider the directed graph $\vec{G}:= (V(G), \vec{E}(G))$ such that $\{u,v\} \in E(G)$ if and only if $(u,v), (v,u) \in \vec{E}(G)$, i.e. in $\vec{G}$ there is exactly one directed edge from $u$ to $v$ and one directed edge from $v$ to $u$ for every edge $\{u, v\} \in E(G)$. We will write $\sum_{u \rightarrow v}$ for the summation over all directed edges $(u, v)$ in $\vec{G}$. The usage of directed edges allows us to talk about a left and right vertex without any ambiguity, while the inclusion of both pairs for any undirected edge ensures we are treating every vertex in such an edge equally.  

To define metric for degree-degree correlations, several empirical density functions for a finite graph \( G = (V(G), E(G)) \) are important. First, the empirical degree density function is given by
\begin{align*}
    f_{G}(k) := \frac{1}{|V(G)|} \sum_{v \in V(G)} \mathbbm{1}_{\{d_{v} = k\}}, \quad \text{for } k=0,1, \dots.
\end{align*}

In addition, the size-biased empirical degree density function is given by
\begin{equation*}
    f^{*}_{G}(k) := \frac{1}{|\vec{E}(G)|} \sum_{u \to v} \mathbbm{1}_{\{d_{u} = k\}} = \frac{1}{\vec{E}(G)|} \sum_{u \in V(G)} k \mathbbm{1}_{\{d_{u} = k\}} = \frac{k |V(G)| f_{G}(k)}{|\vec{E}(G)|}, \quad \text{for } k = 0,1, \dots.
\end{equation*}
The function $ f^{*}_{G}(k) $ represents the fraction of edges for which the start vertex has degree $k$, and it describes the degree distribution with a size bias proportional to the degree. 

We also define the empirical joint degree density function as
\begin{equation*}
    h_{G}(k, \ell) := \frac{1}{|\overset{\rightarrow}{E}(G)|} \sum_{u \rightarrow v} \mathbbm{1}_{\{d_{u} = k, d_{v} = \ell\}}, \quad \text{for } k = 0,1, \dots \text{ and } \ell = 0,1, \dots.
\end{equation*}
The function $ h_{G}(k, \ell) $ represents the fraction of edges connecting a start vertex of degree $k$ to the end vertex of degree $\ell$. It captures the full degree correlation structure between connected nodes in $G$.

Finally, we write $F_G$, $F_G^\ast$ and $H_G$ for the cumulative distribution functions, whose probability mass function corresponds to $f_G$, $f^\ast_G$ and $h_G$, respectively. 

\begin{remark}
We note that the definitions we give for the different degree-degree metrics are valid for any graph $G$, and do not assume any stochasticity. It is only when we consider their limits on sequences of graphs with a local limit that stochasticity comes into play.
\end{remark}

\subsubsection{Known result for Pearson's correlation coefficient}

The first example of how local convergence can be used to derive limit expressions for degree-degree metrics was given in~that $G_n$ converges in probability in the local weak sense to $(G,o)$. Then. Here the limit for Pearson's correlation coefficient was expressed in terms of the degree of the root and that of a randomly sampled neighbor $V$.

We recall that for a finite graph $G$, Pearson's correlation coefficient is given by~\cite{hofstad2014degree}
\begin{equation}\label{def:pearson_r}
     r(G) = \frac{\sum_{u \to v} d_{u} d_{v}- \frac{1}{|\overset{\rightarrow}{E}(G)|} \left(\sum_{v \in V(G)}d^{2}_{v}\right)^{2}}
     {\sum_{v \in V(G)} d^{3}_{v}- \frac{1}{|\overset{\rightarrow}{E}(G)|} \left(\sum_{v \in V(G)}d^{2}_{v}\right)^{2}}.
\end{equation}

\begin{theorem}[\cite{hofstad2024random} Theorem 2.26]\label{thm:pearson}  Let $(G_{n})_{n \ge 1}$ be a sequence of graphs, where $|V(G_{n})|$ tends to infinity and $(d_{o_n}^3)_{n \ge 1}$ is uniformly integrable. Let $(G,o)$ be a random variable in $\mathcal{G}^\ast$ with law $\mu$, such that $\Prob{d_o \ge 1} > 0$. Assume that $G_n$ converges in probability in the local weak sense to $(G,o)$. Then
\begin{equation}
    r(G_{n}) \plim \frac{\mathbb{E}_{\mu}[d_{o}^2d_{V}]- \mathbb{E}_{\mu}[d_{o}^{2}]^{2}/\mathbb{E}_{\mu}[d_{o}]}{\mathbb{E}_{\mu}[d^{3}_{o}]-\mathbb{E}_{\mu}[d_{o}^{2}]^{2}/\mathbb{E}_{\mu}[d_{o}]},
\end{equation}
where $V$ is a uniformly chosen neighbour of root $o$ in graph $(G, o)$.
\end{theorem}

We can now proceed to introduce the other degree-degree metrics and provide our main convergence results for them.

\subsubsection{Spearman's rho}

Spearman's rho~\cite{Spearman1904proof} is an alternative to Pearson's $r$ for measuring correlations. It falls into the category of \emph{rank-correlation} measures as it is based on the rankings of the degrees rather than their actual value. Because of this, it is consistent under far less restrictive conditions than Pearson's $r$~\cite{hofstad2014degree,hoorn2016asymptotic}.  

When applied to degrees in graphs some additional care is needed, as these values are discrete and thus ties can frequently occur. For instance, consider an edge $(u,v) \in \vec{E}$. Then there are at least $d_u$ other edges $(u, v^\prime)$ and thus we will encounter the value $d_u$ at less that many times. There are several different ways to break ties, each leading to a different expression. We will consider the version where ties are broken uniformly at random. This leads to an expression for Spearman's rho as is given in \cite[Equation (3.2)]{hofstad2014degree}. There is, however, a more compact version of Spearman's rho that is asymptotically equivalent and easier to deal with in the mathematical analysis, see \cite[Proposition 5.5]{hoorn2016asymptotic}, \cite[Section 2.3]{hoorn2018generating} or \cite[Section 2.2]{hofstad2014degree}. Therefore, throughout this article, we will use this alternative expression instead of the classical one.  

To define this measure, for every integer $k \ge 0$, let
\begin{equation}\label{F^*_{G}}
    F^{*}_{G}(k) := \frac{1}{|\overset{\rightarrow}{E}(G)|}\sum_{v \in V}d_{v}\mathbbm{1}_{\{d_{v} \le k\}}.
\end{equation}
Similarly, let $\mathcal{F}_{G}(k) = F^{*}_{G}(k)+ F^{*}_{G}(k-1)$ for all $k \ge 0$. Then we define the Spearman's rho degree-degree correlation coefficient by
\begin{equation}\label{def:spearman_rho}
    \rho(G) := \frac{3}{|\overset{\rightarrow}{E}(G)|}\sum_{u \to v}\mathcal{F}_{G}(d_u)\mathcal{F}_{G}(d_v) -3.
\end{equation}

\begin{theorem}\label{thm:spearmans_rho}
Let $(G_n)_{n \ge 1}$ be a sequence of graphs such that $|V(G_n)|$ tends to infinity and $(d_{o_n})_{n \ge 1}$ is uniformly integrable. Let $(G,o)$ be a random variable in $\mathcal{G}^\ast$ with law $\mu$, such that $\Prob{d_o \ge 1} > 0$. Assume that $G_n$ converges in probability in the local weak sense to $(G,o)$. Then
\[
	\rho(G_n) \plim \frac{3}{\Expc{\mu}{d_o}} \Expc{\mu}{d_o \mathcal{F}_\mu^\ast(d_o) \mathcal{F}_\mu^\ast(d_V)} -3,
\]
where $\mathcal{F}_\mu^\ast(k) = F_\mu^\ast(k) + F_\mu^\ast(k-1)$, with
\begin{equation}\label{def:limit_degree_cdf}
	F_\mu^\ast(k) = \frac{\Expc{\mu}{d_o \mathbbm{1}_{d_o \le k}}}{\Expc{\mu}{d_o}},
\end{equation}
and $V$ is a neighbor of $o$ chosen uniformly at random.
\end{theorem}

Note that indeed convergence of Spearman's $\rho$ holds whenever $(d_{o_n})_{n \ge 1}$ is uniformly integrable, which is far less restrictive than requiring uniform integrability of $(d_{o_n}^3)_{n \ge 1}$ needed for Pearson's $r$.

\subsubsection{Kendall's tau}

Another rank-correlation measure is Kendall's tau~\cite{Kendall1938new}. It computes the difference between concordant and dicordant pairs of joint observations, normalized by the total number of pairs, see~\cite[Section 4.4]{hoorn2016asymptotic}. Similar to Spearman's rho we also have alternative version of Kendall's tau, see~\cite[Section 5.3.1]{hoorn2016asymptotic}, which we will consider in this article.

We first define for any pair of integers $k, \ell \ge 0$,
\begin{equation}
	\mathcal{H}_G(k,\ell) = H_G(k,\ell) + H_G(k-1,\ell) + H_G(k,\ell-1) + H_G(k-1,\ell-1),
\end{equation}
where 
\[
	H_G(k,\ell) = \frac{1}{|\vec{E}(G)|} \sum_{u \to v} \mathbbm{1}_{d_u \le k, d_v \le \ell}
\] 
is the joint degree distribution. Then Kendall's tau for degree-degree correlations is given by
\begin{equation}\label{def:kendall_tau}
	\tau(G) = \frac{1}{|\vec{E}(G)|} \sum_{u \to v} \mathcal{H}_G(d_u, d_v) - 1.
\end{equation}

\begin{theorem}\label{thm:kendalls_tau}
Let $(G_n)_{n \ge 1}$ be a sequence of rooted graphs and assume that $(d_{o_n})_{n \ge 1}$ is uniformly integrable. Let $(G,o)$ be a random variable in $\mathcal{G}^\ast$ with law $\mu$, such that $\Prob{d_o \ge 1} > 0$. Assume that $G_n$ converges in probability in the local weak sense to $(G,o)$. Then
\[
	\tau(G_n) \plim \frac{\Expc{\mu}{d_o\mathcal{H}_\mu(d_o,d_V)}}{\Expc{\mu}{d_o}} - 1,
\]
where $\mathcal{H}_\mu(k, \ell) = H_\mu(k,\ell) + H_\mu(k-1,\ell) + H_\mu(k,\ell-1) + H_\mu(k-1,\ell-1)$, with
\[
	H_\mu(k,\ell) = \frac{\Expc{\mu}{d_o \mathbbm{1}_{d_o \le k} \mathbbm{1}_{d_V \le \ell}}}{\Expc{\mu}{d_o}},
\]
and $V$ is a neighbor of $o$ chosen uniformly at random.
\end{theorem}

\subsubsection{Degree-degree distance}

In addition to classical (rank-based) correlation metrics, one can also measure degree-degree correlations by looking at the differences between the degrees on both sides of an edge. This was recently proposed independently in~\cite{farzam2020degree} and~\cite{zhou2020power}, using the notion of degree-degree distance\footnote{This is called degree difference in~\cite{farzam2020degree}}.  

\begin{definition}\label{def:degree_distance}
For any graph $G$ and monotone function $g : \mathbb{N}_0 \to \mathbb{R}_+$, define the degree distance as
\begin{equation}\label{eq:degree_distance}
	\delta(G) = \frac{1}{|\vec{E}(G)|} \sum_{u \to v} \left|g(d_u) - g(d_v)\right|.
\end{equation}
\end{definition}

The version of this measure for $g(x) = x$ was considered in~\cite{farzam2020degree} while~\cite{zhou2020power} used $g(x) = \log(x)$. Here we establish the limit degree-degree distance under local convergence in probability.

\begin{theorem}\label{thm:degree_distance}
Let $g : \mathbb{N}_0 \to \mathbb{R}_+$ be a monotone function and $(G_n)_{n \ge 1}$ a sequence of rooted graphs such that $(d_{o_n} g(d_{o_n}))_{n \ge 1}$ is uniformly integrable. Let $(G,o)$ be a random variable in $\mathcal{G}^\ast$ with law $\mu$, such that $\Prob{d_o \ge 1} > 0$. Assume that $G_n$ converges in probability in the local weak sense to $(G,o)$. Then
\[
	\delta(G_n) \plim \frac{\Expc{\mu}{d_o \left|g(d_o) - g(d_V)\right|}}{\Expc{\mu}{d_o}}.
\]
\end{theorem}

It should be noted that the conditions required on the degrees depends on the function $g$ used. For the case $g(x) = x$ this boils down to the basic uniform integrability of $(d_{o_n}^2)_{n \ge 1}$.

\subsubsection{Average nearest neighbor degree and rank}

We now move our attention to two measures for degree-degree correlations which capture the local structure of a network. Compared to Spearman's rho, Kendall's tau and degree-degree distance, which provide a single global value summarizing the overall assortativity or disassortativity of the entire network, these two measures offer a more detailed view. In particular, these measures show for each integer $k$ how nodes of degree $k$ connect to nodes of various degrees, which helps us understand particular behaviors that might be hidden in other global measures we considered.  

We start with the Average Nearest Neighbor Degree (ANND), which measures the average of the degrees of all nodes connected to a given node of degree $k$, properly normalize. It is defined as
\begin{equation}\label{def:annd}
	\Phi_G(k) = \mathbbm{1}_{f_G(k) > 0} \frac{\sum_{\ell > 0} \ell \, h_G(k, \ell)}{f_G^\ast(k)}, \quad \text{for } k = 1, 2, \dots.
\end{equation}

\begin{theorem}\label{thm:annd}
Let $(G_n)_{n \ge 1}$ be a sequence of rooted graphs and assume that $(d_{o_n}^2)_{n \ge 1}$ is uniformly integrable. Let $(G,o)$ be a random variable in $\mathcal{G}^\ast$ with law $\mu$, such that $\Prob{d_o \ge 1} > 0$. Assume that $G_n$ converges in probability in the local weak sense to $(G,o)$. Then for all $k \ge 0$ such that $\Prob{d_o = k} > 0$,
\[
	\phi_{G_n}(k) \plim \Expc{\mu}{d_V| d_o = k}.
\]
\end{theorem}

Observe that the ANND compares degrees directly and thus, like Pearson's $r$ needs additional moment conditions for convergence to hold, in this case uniform integrability of $d_{o_n}^2$. However, we can apply a similar approach as for Spearman's rho, and use the ranks of the degrees instead, or similarly by apply $F_G^\ast$ to the degrees. This leads to what is known as the Average Nearest Neighbor Rank (ANNR), cf.~\cite{yao2018average},
 \begin{equation}\label{ANNR}
     \Theta_{G}(k) :=  \mathbbm{1}_{\{f_{G}(k)>0\}} \frac{\sum_{l>0} F^\ast_{G}(l) \, h_{G}(k,l)}{f^\ast_{G}(k)}, \quad \text{for } k = 1,2, \dots.
 \end{equation}

Working with ranks instead of the values of the degrees yield a convergence results that only requires uniform integrability of the degrees.

\begin{theorem}\label{thm:annr}
Let $(G_n)_{n \ge 1}$ be a sequence of rooted graphs and assume that $(d_{o_n})_{n \ge 1}$ is uniformly integrable. Let $(G,o)$ be a random variable in $\mathcal{G}^\ast$ with law $\mu$, such that $\Prob{d_o \ge 1} > 0$. Assume that $G_n$ converges in probability in the local weak sense to $(G,o)$. Then for all $k \ge 0$ such that $\Prob{d_o = k} > 0$,
\[
	\theta_{G_n}(k) \plim \Expc{\mu}{F_\mu^\ast(d_V)| d_o = k},
\]
where $F_\mu^\ast$ is defined as in~\eqref{def:limit_degree_cdf}.
\end{theorem}

\section{Applications}\label{sec:applications}

To showcase our results we apply them to two well-known random graph models. In particular, for both models we provide explicit expressions for the limit of the Average Nearest Neighbor Degree. The first are rank-1 inhomogeneous random graphs, which will serve as an example of graphs with neutral mixing. After that, we study the behavior of degree-degree correlations metrics in Random Geometric Graphs, which are seen as examples of graphs with assortative mixing, although not many results are known. We provide formal results confirming that these graphs have assortative mixing.

\subsection{Rank-1 Inhomogeneous Random Graphs}

Let $W$ be a non-negative random variable and $n \in \mathbb{N}$. Then the rank-1 inhomogeneous random graph $\mathrm{IRG}_n(W)$ on $n$ nodes is constructed by considering a sequence $W_1, \dots, W_n$ of i.i.d. random variables with distribution equal to $W$ and connect each pair of nodes $i$ and $j$ independently with probability
\[
	p_{ij} := \min\left\{\frac{W_i W_j}{n}, 1\right\}.
\]

This model is a specific instance of a more general class of inhomogeneous random graphs (see~\cite{bollobas2007phase} or~\cite{hofstad2024random}), which are in turn generalizations of the Chung-Lu model~\cite{chung2002average}. 

In particular, it is known (see for example~\cite[Theorem 3.14]{hofstad2024random}) that these graphs converge locally to a Galton-Watson tree with the root having off-spring distributed as $\mathrm{Po}(W)$ and each other individual have independent off-spring that is distributed according to the sized-biased distribution $\mathrm{Po}(W)^\ast$, where
\[
	\Prob{\mathrm{Po}(W)^\ast = k} = \frac{k\Prob{\mathrm{Po}(W) = k}}{\Exp{W}}, \quad \text{for } k = 0,1,2,\dots
\]

From this we immediately deduce for the limit graph $(G,o)$ that $d_o \stackrel{d}{=} \mathrm{Po}(W)$ and $d_V \stackrel{d}{=} 1 + d_o^\ast$ has the size-biased distribution and is independent from $d_o$. The fact that $d_o$ and $d_V$ are independent immediately implies that the limits of $\rho(G_n)$ and $\tau(G_n)$ are zero, for sequences of rank-1 IRGs $\mathrm{IRG}(W,n)$. Our first application result shows that the average nearest neighbor degree converges to a fixed constant for every $k$, resembling the neutral mixing.

\begin{proposition}[Average nearest neighbor degree IRGs]\label{prop:annd_IRG}
Let $(\mathrm{IRG}_n(W))_{n \ge 1}$ be a sequence of rank-1 IRGs such that $(d_{o_n}^2)_{n \ge 1}$ is uniformly integrable. Then for any $k \ge 1$,
\[
	\phi_{G_n}(k) \plim 1 + \frac{\Exp{W^2}}{\Exp{W}}.
\]
\end{proposition}

Note that a sufficient condition for uniform integrability of $(d_{o_n}^2)_{n \ge 1}$ is $\Exp{W^{2+\delta}} < \infty$ for some $0 < \delta <1$.

\begin{remark}
Note that $1 + \Exp{W^2}/\Exp{W} = \Exp{d_o^2}/\Exp{d_o}$. Hence, the result in Proposition~\ref{prop:annd_IRG} reflects the one obtained in~\cite{yao2018average} for the repeated and erased configuration model (see Theorem 5.5 and Theorem 5.8). This is expected as both these models have the same Galton-Watson tree as their local limit~\cite{hofstad2016random} and this is the only relevant part for establishing the limit of degree-degree metrics. Therefore, the limit established in~\cite{yao2018average} follows from our main result and Proposition~\ref{prop:annd_IRG}. Nevertheless, it should be noted that the results in~\cite{yao2018average} were established under different assumptions on the convergence of the empirical (joint) degree distributions (see Assumptions 4.1 and 4.2) and also include bounds on the speed of convergence. Since there is no direct implication link between these assumptions and local convergence, which we assume in this work, both results should be considered separately.  
\end{remark}

\subsection{Random Geometric Graphs}\label{ssec:random_geometric_graphs}

While the application of our results to rank-1 inhomogeneous random graphs yields mostly known results, we now move to a class of models for which degree-degree correlations are not extensively studied: random geometric graphs. 

Let $\mathbb{T}_n^d$ denote the $d$-dimensional torus of volume $n$, i.e. the box $I_n^d := [-n^{1/d}/2, n^{1/d}/2]^d$ with the boundaries identified. Furthermore, let $p \in (0,1]$ and $R > 0$. Then the random geometric graph with $n$ vertices $\mathrm{RGG}_n(p,R)$ is constructed by placing $n$ points $X_1, \dots, X_n$ uniformly at random in $I_n^d$ and connecting two points $X_i$ and $X_j$ independently with probability
\[
	p_{ij} = p \mathbbm{1}_{\|X_i - X_j\|_n \le R},
\]
where $\| \cdot \|_n$ is the torus metric.

The classical case of random geometric graphs is when $p = 1$.

As we let $n$ tend to infinity, the box $I_n^d$ will blow up to $\mathbb{R}^d$. Since nodes only care about a fixed neighborhood for establishing edges, the fact that $\mathbb{T}_n^d$ has no boundary will play a diminishing role as $n \to \infty$, while a unit-rate Poisson process will take the role of the `'infinite version" of placing $n$ point uniformly at random. So we would expect that the local limit of random geometric graphs will consist of a graph whose nodes correspond to a unit-rate Poisson process on $\mathbb{R}^d$ and with connection probability
\[
	p_{ij} = p \mathbbm{1}_{\|X_i - X_j\| \le R},
\]
where we now use the normal Euclidean metric. The root will then simply be the origin of $\mathbb{R}^d$

The following result establishes this intuition. It is a specific instance of a more general convergence result from~\cite{hofstad2023local} applied to the case of random geometric graphs.

\begin{theorem}
Let $(\mathrm{RGG}_n(p,R))_{n \ge 1}$ be a sequence of random geometric graphs with connection radius $R > 0$ and connection probability $p$, and let $(G_\infty(R),o)$ be defined as above. Then $G_n \to (G_\infty(R), o)$.
\end{theorem}

Now that we have the local limit, we can apply our results to compute degree-degree measures for random geometric graphs. In the remainder of this section we write $v_d := \pi^{d/2}/ \Gamma(1 + d/2)$ to denote the volume of the unit ball in $\mathbb{R}^d$ and denote by $\omega_d(R) := v_d R^d$ the volume of the ball in $\mathbb{R}^d$ with radius $R$. In addition, we define
\begin{equation}
	p_{\text{conn}} = \frac{1}{\omega_d(R)^2} \iint_{x,y \in B(0,R)} \mathbbm{1}_{\|x-y\| \le R} \, dx \, dy. 
\end{equation}

We start with giving a result for Pearson's correlation coefficient.

\begin{proposition}[Pearson's limit in RGGs]\label{prop:pearson_RGGs}
Let $(\mathrm{RGG}_n(p,R))_{n \ge 1}$ be a sequence of random geometric graphs with connection radius $R > 0$ and connection probability $p$. Then
\[
		r(\mathrm{RGG}_n(p,R)) \plim p\,p_{\mathrm{conn}}.
\]
\end{proposition}

\begin{remark}
From Proposition~\ref{prop:pearson_RGGs} we make two important observations. First note that the limit of $r(\mathrm{RGG}_n(p,R))$ does not depend on $p$. This can be explained by looking at the graph $\mathrm{RGG}_n(p,R)$ as being constructed by first generating the graph $\mathrm{RGG}_n(1,R)$ and then keeping each edge independently with probability $p$. Since this procedure treats every edge independently and equally, it should not influence the dependency of the degrees of connected nodes.

Second, $p_{\text{conn}}$ is strictly positive. Hence, random geometric graphs are an example of assortative graphs, i.e. those with positive degree-degree correlations. This is not unexpected as the geometric aspect to the connection rule creates the presence of joint neighbors, which in turn establishes a positive correlation between the degrees of two connected nodes. However, as far as we know, this is the first explicit expression for the limit of this measure.
\end{remark}

We also obtain the limit of the average nearest neighbor degree.

\begin{proposition}[ANND limit in RGGs]\label{prop:annd_RGGs}
Let $(\mathrm{RGG}_n(p,R))_{n \ge 1}$ be a sequence of random geometric graphs with connection radius $R > 0$ and connection probability $p$. Then for any $k \ge 1$,
\[
	\phi_{\mathrm{RGG}_n(p,R)}(k) \plim 1 + p\omega_d(R)(1 - p\,p_\mathrm{conn}) + p(k-1)p_\mathrm{conn}.
\]
\end{proposition}

This result shows that the average degree of the nearest neighbor scale linearly with the degree $k$ of the root, making RGGs a clear example of an assortative random graph model.

\begin{remark}
Our results for random geometric graphs complement and extend those in the recent work~\cite{kaufmann2025assortativity}. Here the degree distribution of a node on a randomly sampled edge was computed and used to argue why these graphs have assortative mixing. In contrast, our results show the assortative mixing by proving explicit values for two degree-degree correlation measures.
\end{remark}

\begin{remark}
We can also apply our results for Spearman's $\rho$, Kendall's $\tau$ and the average nearest neighbor rank to the limit of random geometric graphs. Unfortunately, this will not yield an insightful and compact expression. For example, for Spearman's $\rho$ the limit will involve terms of the form
\[
	\Exp{F_d(X+Z-1)F_d(Y+Z-1)},
\]
where $F_d$ denotes the cdf of a Poisson random variable with mean $\omega_d(R)$ and $X, Y, Z$ are Poisson random variable whose parameters represent the volumes of the intersection of two balls and those of their two disjoint parts. While it is possible to write out this expression fully, it does not yield any particularly insightful closed formula. We thus omit these computations here.
\end{remark}

\section{Proofs of main results}\label{sec:main_proofs}

\subsection{General idea and approach}

The first important thing to note is that the definition of local convergence in probability (Definition~\ref{def:local_convergence_probability}) is equivalent to
\[
	\CExp{h(G_n,o_n)}{G_n} \plim \Expc{\mu}{h(G,o)},
\]
for any continuous bounded function $h : \mathcal{G}_\star \to \mathbb{R}$~\cite{hofstad2024random}. This implication will be used often in our proofs.

Many proofs will boil down to recognizing the expression as the expected value of sum function $\phi$ evaluated on $d_{o_n}$ and $d_{V_n}$, conditioned on the graph $G_n$. Once this is achieved, the following technical lemma can be used to obtain the convergence to the appropriate limit under the right uniform integrability conditions.

\begin{lemma}\label{lem:local_convergence_tool}
Let $(G_n)_{n \ge 1}$ be a sequence of graphs such that $|V(G_n)|$ tends to infinity and $G_n$ converges locally in probability to $(G,o)$ with law $\mu$. Let $\phi : \mathbb{N}_0 \times \mathbb{N}_0 \to \mathbb{R}_+$ be a measurable function such that the sequence $(\phi(d_{o_n},d_{V_n})_{n \ge 1}$ is uniformly integrable, with $o_n$ a uniform random node in $G_n$ and $V_n$ a uniform neighbor. Then
\[
	\CExp{\phi(d_{o_n},d_{V_n})}{G_n} \plim \Expc{\mu}{\phi(d_o, d_V)},
\] 
where $V$ is a uniform neighbor of $o$.
\end{lemma}

\begin{proof}
For any $K \ge 0$ we have
\begin{align*}
	&\hspace{-20pt} \left|\CExp{\phi(d_{o_n},d_{V_n})}{G_n} - \Expc{\mu}{\phi(d_o, d_V)}\right|\\
	&\le \left|\CExp{\phi(d_{o_n},d_{V_n})\mathbbm{1}_{\phi(d_{o_n},d_{V_n}) \le K}}{G_n} 
		- \Expc{\mu}{\phi(d_o, d_V) \mathbbm{1}_{\phi(d_o,d_V) \le K}}\right| \\
	&\hspace{10pt}+ \CExp{\phi(d_{o_n},d_{V_n})\mathbbm{1}_{\phi(d_{o_n},d_{V_n}) > K}}{G_n} \\
	&\hspace{10pt}+ \Expc{\mu}{\phi(d_o, d_V) \mathbbm{1}_{\phi(d_o,d_V) > K}}.
\end{align*}

Since $(G,o) \mapsto \phi(d_{o},d_{V})\mathbbm{1}_{\phi(d_{o},d_{V}) \le K}$ is bounded and continuous\footnote{In the topology on $\mathcal{G}_\star$.} the first term converges to zero in probability, as $n \to \infty$, due to the fact that $(G_n)_{n\ge 1}$ converges locally in probability to $(G,o)$. In particular, this implies that
\[
	\limsup_{K \to \infty} \limsup_{n \to \infty} 
	\Prob{\left|\CExp{\phi(d_{o_n},d_{V_n})\mathbbm{1}_{\phi(d_{o_n},d_{V_n}) \le K}}{G_n} 
			- \Expc{\mu}{\phi(d_o, d_V) \mathbbm{1}_{\phi(d_o,d_V) \le K}}\right| > \varepsilon} = 0.
\]

For the second term, we use the uniform integrability of $\phi(d_{o_n},d_{V_n})$ and Markov's inequality to conclude that
\[
	\limsup_{K \to \infty} \limsup_{n \to \infty} 
		\Prob{\CExp{\phi(d_{o_n},d_{V_n})\mathbbm{1}_{\phi(d_{o_n},d_{V_n}) > K}}{G_n} > \varepsilon} = 0.
\]

Finally, for the third term (which does not depend on $n$) we have 
\[
	\limsup_{K \to \infty} \Expc{\mu}{\phi(d_o, d_V) \mathbbm{1}_{\phi(d_o,d_V) > K}} = 0.
\]

Putting this together we conclude that
\begin{align*}
	&\hspace{-20pt}\limsup_{n \to \infty} \Prob{\left|\CExp{\phi(d_{o_n},d_{V_n})}{G_n} 
		- \Expc{\mu}{\phi(d_o, d_V)}\right| > \varepsilon} \\
	&= \limsup_{K \to \infty} \limsup_{n \to \infty} \Prob{\left|\CExp{\phi(d_{o_n},d_{V_n})}{G_n} 
			- \Expc{\mu}{\phi(d_o, d_V)}\right| > \varepsilon} = 0.
\end{align*}
\end{proof}

In addition, we will also make use of several basic facts concerning convergence in probability, summarized in the following lemma. For completeness, we include the proof in the Appendix~\ref{sec:proof_appendix}.

\begin{lemma}\label{lem:basic_convergence_results}
Let $(G_n)_{n \ge 1}$ be a sequence of graphs such that $|V(G_n)|$ tends to infinity, and $G_n$ converges locally in probability to $(G,o)$ with law $\mu$. Then
\begin{enumerate}
\item $\displaystyle \sup_{k \ge 0} \left|F^\ast_{G_n}(k) - F_\mu^\ast(k)\right| \plim 0$;
\item $\displaystyle \sup_{k, \ell \ge 0} \left|H_{G_n}(k,\ell) - H_\mu(k,\ell)\right| \plim 0$;
\item $\displaystyle \CExp{F_{G_n}^\ast(d_{o_n})}{G_n} \plim \Expc{\mu}{F_\mu^\ast(d_o)}$;
\item If, in addition, $(d_{o_n})_{n \ge 1}$ is uniformly integrable and $\Prob{d_o \ge 1} > 0$, then 
\[
	\frac{2|V(G_n)|}{|\vec{E}(G_n)|} \plim \frac{1}{\Expc{\mu}{d_o}}.
\]
\end{enumerate}
\end{lemma}

With this setup, we are ready to provide the proofs of our main results.

\subsection{Spearman's rho}

\begin{proof}[Proof of Theorem~\ref{thm:spearmans_rho}]
We have to prove that
\begin{equation}\label{eq:spearman_desired_cnvergence}
	\frac{1}{|\vec{E}(G_{n})|}\sum_{u \to v}\mathcal{F}_{G_{n}}(d_u)\mathcal{F}_{G_{n}}(d_v) \xrightarrow{\mathbb{P}} \frac{\mathbb{E}_{\mu}[d_{o}\mathcal{F}_\mu(d_{o})\mathcal{F}_\mu(d_{V})]}{\mathbb{E}_{\mu}[d_{o}]}.
\end{equation}

To this end, we replace each instance of $\mathcal{F}_{G_n}$ with $\mathcal{F}_\mu$ and observe that 
\begin{align*}
        \frac{1}{|\vec{E}(G_{n})|}\sum_{u \to v}\mathcal{F}_\mu(d_u)\mathcal{F}_\mu(d_v) 
        &= \frac{2}{|\vec{E}(G_{n})|} \sum_{u \in V(G_{n})}\mathcal{F}_\mu(d_{u})
        	\sum_{v \in V(G_{n}), v \sim u} \mathcal{F}_\mu(d_{v}) \\
        &= \frac{2|V(G_n)|}{|\vec{E}(G_{n})|} \frac{1}{|V(G_n)|}\sum_{u \in V(G_{n})}d_{u}\mathcal{F}_\mu(D_{u})
        	\left(\frac{1}{d_{u}}\sum_{v \in V(G_{n}), v \sim u} \mathcal{F}_\mu(d_{v})\right) \\
        &= \frac{2|V(G_n)|}{|\vec{E}(G_{n})|} \CExp{d_{o_{n}}\mathcal{F}_\mu(d_{o_{n}}) \mathcal{F}_\mu(d_{V_n})}{G_n},
\end{align*}
where $V_n$ is a uniformly neighbour of $o_{n}$ in $G_n$. 

By Lemma~\ref{lem:basic_convergence_results} $2|V(G_n)|/|\vec{E}(G_n)| \plim \Expc{\mu}{d_o}^{-1}$. Next, since the sequence $\left(d_{o_{n}}\right)_{n \ge 1}$ is uniformly integrable and $d_{o_{n}}\mathcal{F}_\mu(d_{o_{n}}) \mathcal{F}_\mu(d_{V_n}) \le 4 d_{o_{n}}$, we conclude that $\left( d_{o_{n}}\mathcal{F}_\mu(d_{o_{n}}) \mathcal{F}_\mu(d_{V_n})\right)_{n \ge 1}$ is also uniformly integrable. Using Lemma~\ref{lem:local_convergence_tool}, this then implies that
\begin{equation}\label{eq:spearman_result_almost}
	\frac{|V(G_n)|}{|\vec{E}(G_{n})|} \CExp{d_{o_{n}}\mathcal{F}_\mu(d_{o_{n}}) \mathcal{F}_\mu(d_{V_n})}{G_n} 
	\plim \frac{\Expc{\mu}{d_{o}\mathcal{F}_\mu(d_{o})\mathcal{F}_\mu(d_{V})}}{\Expc{\mu}{d_o}},
\end{equation}
where $V$ is a neighbor of $o$ sampled uniformly at random.

To finish the proof, we recall that the difference between~\eqref{eq:spearman_desired_cnvergence} and~\eqref{eq:spearman_result_almost} is that $\mathcal{F}_{G_n}$ is replaced with $\mathcal{F}_\mu$. Therefore, the result follows if we can show that
\begin{equation}\label{eq:spearman_approximation}
    \frac{1}{|\vec{E}(G_{n})|}\left|\sum_{u \to v}\mathcal{F}_{G_{n}}(d_u)\mathcal{F}_{G_{n}}(d_v) 
    - \sum_{u \to v}\mathcal{F}_\mu(d_u)\mathcal{F}_\mu(d_v)\right| \plim 0.   
\end{equation}

Lemma~\ref{lem:basic_convergence_results} implies that for all $k \ge 0$ 
\begin{equation*}
        \mathcal{F}_{G_{n}}(k) = F^{*}_{G_{n}}(k)+ F^{*}_{G_{n}}(k-1) \xrightarrow{\mathbb{P}} F_\mu^\ast(k) + F_\mu^\ast(k-1) := \mathcal{F}_\mu(k).
    \end{equation*} 

For $n \ge 1$ and $k \ge 0$ let $X_{n}(k) := \mathcal{F}_{G_{n}}(k)-\mathcal{F}_\mu(k)$. Then, 
\begin{align*}
	\sum_{u \to v}\mathcal{F}_{G_{n}}(d_u)\mathcal{F}_{G_{n}}(d_v) 
	&= \sum_{u \to v}\left(\mathcal{F}_\mu(d_u)+ X_{n}(d_u) \right) \left(\mathcal{F}_\mu(d_v)+ X_{n}(d_v)\right) \\
	&= \sum_{u \to v}\mathcal{F}_\mu(d_u)\mathcal{F}_\mu(d_v) + \sum_{u \to v}X_{n}(d_u)\mathcal{F}_\mu(d_v) \\
	&\hspace{10pt}+ \sum_{u \to v}\mathcal{F}_\mu(d_u)X_{n}(d_v) + \sum_{u \to v}X_{n}(d_u)X_{n}(d_v).
\end{align*}
Moreover, since $0 \le \mathcal{F}_\mu(k) \le 2$, for all $u,v \in V(G_{n})$
\begin{equation*}
    \begin{aligned}
    	0 \le X_{n}(d_u)\mathcal{F}_{G}(D_{u}) \le 2\sup_{m \ge 0} 
        	\left| \mathcal{F}_{G_{n}}(m)-\mathcal{F}_\mu(m) \right|, \\
    	0 \le X_{n}(d_u)X_{n}(d_v) \le \left(\sup_{m \ge 0} 
    		\left| \mathcal{F}_{G_{n}}(m)-\mathcal{F}_\mu(m) \right|\right)^{2}. 
    \end{aligned}
\end{equation*}

We then obtain the following upper bound
\begin{align*}
	&\hspace{-20pt} \frac{1}{|\vec{E}(G_{n})|}\left|\sum_{u \to v}\mathcal{F}_{G_{n}}(d_u)\mathcal{F}_{G_{n}}(d_v) 
		- \sum_{u \to v}\mathcal{F}_\mu(d_u)\mathcal{F}_\mu(d_v)\right| \\
	&\le  \frac{1}{|\vec{E}(G_{n})|}\sum_{u \to v}X_{n}(d_u)\mathcal{F}_\mu(d_v)
		+  \frac{1}{|\vec{E}(G_{n})|}\sum_{u \to v}\mathcal{F}_\mu(d_u)X_{n}(d_v)\\
	&\hspace{10pt}+  \frac{1}{|\vec{E}(G_{n})|}\sum_{u \to v}X_{n}(d_u)X_{n}(d_v)\\
	&\le 4\sup_{m \ge 0} \left\vert \mathcal{F}_{G_{n}}(m)-\mathcal{F}_\mu(m) \right\vert
		+ \left(\sup_{m \ge 0} \left\vert	\mathcal{F}_{G_{n}}(m)-\mathcal{F}_\mu(m) \right\vert\right)^{2}.
	\numberthis \label{eq:spearman_final_term}
\end{align*}

By Lemma~\ref{lem:basic_convergence_results} we know that
\begin{equation}
    \sup_{m \ge 0} \left\vert F^\ast_{G_{n}}(m)-F_\mu^\ast(m)\right\vert \plim 0.
\end{equation}
Since $\mathcal{F}_{G_{n}}$ and $\mathcal{F}_\mu$ are linear combinations of $F^\ast_{G_{n}}$ and $F_\mu^\ast(m)$, respectively, last two terms in~\eqref{eq:spearman_final_term} converge to zero in probability, which establishes~\eqref{eq:spearman_approximation} and finishes the proof.
\end{proof}

\subsection{Kendall's tau}

\begin{proof}[Proof of Theorem~\ref{thm:kendalls_tau}]

We follow the same strategy as for the proof of Theorem~\ref{thm:spearmans_rho}. Hence, we need to show that
\[
	\frac{1}{|\vec{E}(G_n)|} \sum_{u \to v} \mathcal{H}_{G_n}(d_u, d_v) 
	\plim \frac{\Expc{\mu}{d_o \mathcal{H}_\mu(d_o,d_V)}}{\Expc{\mu}{d_o}}.
\]

Again, by replacing $\mathcal{H}_{G_n}$ in the term on the left hand side with $\mathcal{H}_\mu$ we get that
\begin{align*}
	\frac{1}{|\vec{E}(G_n)|} \sum_{u \to v} \mathcal{H}_\mu(d_u, d_v) 
	&= \frac{2|V(G_n)|}{|\vec{E}(G_n)|} \frac{1}{|V(G_n)|} \sum_{u \in V(G_n)} d_u 
		\sum_{v \in V(G_n), \, u \sim v} \frac{1}{d_u} \mathcal{H}_\mu(d_u, d_v) \\
	&= \frac{2|V(G_n)|}{|\vec{E}(G_n)|} \CExp{d_{o_n} \mathcal{H}_\mu(d_{o_n}, d_{V_n}) }{G_n},
\end{align*}
with $V_n$ a uniform random neighbor of $o_n$. Using Lemma~\ref{lem:basic_convergence_results} we have that $2|V(G_n)|/|\vec{E}(G_n)| \plim \Expc{\mu}{d_o}^{-1}$. Moreover, since $\mathcal{H}_\mu(k,\ell) \le 4$ the sequence $(d_{o_n} \mathcal{H}_\mu(d_{o_n}, d_{V_n}))_{n \ge 1}$ is uniformly integrable. This then implies that
\[
	\frac{1}{|\vec{E}(G_n)|} \sum_{u \to v} \mathcal{H}_\mu(d_u, d_v) 
	\plim \frac{\Expc{\mu}{d_o \mathcal{H}_\mu(d_o,d_V)}}{\Expc{\mu}{d_o}},
\]
and hence we are left to show that
\begin{equation}\label{eq:kendall_approximation}
	\frac{1}{|\vec{E}(G_n)|}\left|\sum_{u \to v} \mathcal{H}_{G_n}(d_u, d_v) - \sum_{u \to v} \mathcal{H}_\mu(d_u, d_v)\right|
	\plim 0.
\end{equation}

This follows readily since
\[
	\frac{1}{|\vec{E}(G_n)|}\left|\sum_{u \to v} \mathcal{H}_{G_n}(d_u, d_v) 
		- \sum_{u \to v} \mathcal{H}_\mu(d_u, d_v)\right|
	\le \sup_{k, \ell \ge 0} \left|\mathcal{H}_{G_n}(k,\ell) - \mathcal{H}_\mu(k,\ell)\right|
\]
and the right hand side converges to zero in probability by Lemma~\ref{lem:basic_convergence_results}.

\end{proof}

\subsection{Degree-degree distance}

\begin{proof}[Proof of Theorem~\ref{thm:degree_distance}]

Following the approach from the previous proofs, we first write
\begin{align*}
	\delta(G_n) &= \frac{2|V(G_n)|}{|\vec{E}(G_n)|} \frac{1}{|V(G_n)|} 
		\sum_{u \in V(G_n)} d_u \sum_{v \in V(G_n)} \frac{1}{d_u} |g(d_u) - g(d_v)| \\
	&= \frac{2|V(G_n)|}{|\vec{E}(G_n)|} \CExp{d_{o_n} |g(d_{o_n}) - g(d_{V_n})|}{G_n}, \\
\end{align*}
where $V_n$ is a uniform random neighbor of $o_n$. Again, $2|V(G_n)|/|\vec{E}(G_n)| \plim \Expc{\mu}{d_o}^{-1}$. Moreover, we observe that
\begin{align*}
	\CExp{d_{o_n} |g(d_{o_n}) - g(d_{V_n})|}{G_n} 
	&= \frac{1}{2|V(G_n)|} \sum_{u \to v} |g(d_u) - g(d_v)|\\
	&\le \frac{2}{|V(G_n)|} \sum_{u \in V(G_n)} d_u g(d_u) = 2\CExp{d_{o_n} g(d_{o_n})}{G_n}.
\end{align*}
Therefore, since by assumption $(d_{o_n} g(d_{o_n}))_{n \ge 1}$ is uniform integrable, so is $(d_{o_n} |g(d_{o_n}) - g(d_{V_n})|)_{n \ge 1}$. Thus Lemma~\ref{lem:local_convergence_tool} implies that 
\[
	\CExp{d_{o_n} |g(d_{o_n}) - g(d_{V_n})|}{G_n} \plim \Expc{\mu}{d_o |g(d_o) - g(d_V)|},
\]
which finishes the proof.
\end{proof}

\subsection{Average nearest neighbor degree and rank}

\begin{proof}[Proof of Theorem~\ref{thm:annd}]

We start with rewriting the expression of $\psi_G$ in~\eqref{def:annd} as follows
\begin{align*}
	\mathbbm{1}_{f_{G_n}(k) > 0} \frac{\sum_{\ell > 0} \ell \, h_{G_n}(k,\ell)}{f_{G_n}^\ast(k)} 
	&=  \frac{\mathbbm{1}_{f_{G_n}(k) > 0} }{f_{G_n}^\ast(k)} \frac{1}{|\vec{E}(G_n)|}
		\sum_{\ell > 0} \sum_{u \to v} d_v \mathbbm{1}_{d_u = k, d_v = \ell} \\
	&= \frac{\mathbbm{1}_{f_{G_n}(k) > 0} }{f_{G_n}^\ast(k)} \frac{1}{|\vec{E}(G_n)|} 
		\sum_{u \to v} d_v \mathbbm{1}_{d_u = k} \\
	&= \frac{\mathbbm{1}_{f_{G_n}(k) > 0} }{f_{G_n}^\ast(k)} \frac{1}{|\vec{E}(G_n)|} 
		\sum_{u \in V(G_n)} d_u \mathbbm{1}_{d_u = k}	\sum_{v \in V(G_n), \, u \sim v} \frac{1}{d_u} d_v\\
	&= \frac{k \mathbbm{1}_{f_{G_n}(k) > 0} }{f_{G_n}^\ast(k)} \frac{1}{|\vec{E}(G_n)|} 
		\sum_{u \in V(G_n)} \mathbbm{1}_{d_u = k}	\sum_{v \in V(G_n), \, u \sim v} \frac{1}{d_u} d_v\\
	&= \frac{\mathbbm{1}_{f_{G_n}(k) > 0}}{f_{G_n}(k)} \frac{1}{V(G_n)}
		\sum_{u \in V(G_n)} \mathbbm{1}_{d_u = k}	\sum_{v \in V(G_n), \, u \sim v} \frac{1}{d_u} d_v\\
	&= \frac{\mathbbm{1}_{f_{G_n}(k) > 0}}{f_{G_n}(k)} \CExp{\mathbbm{1}_{d_{o_n} = k} d_{V_n}}{G_n}.
\end{align*}
Here, for the second to last step we used that
\[
	f_{G_n}^\ast(k) = \frac{|V(G_n)|}{|\vec{E}(G_n)|} k f_{G_n}(k).
\]

Since $\mathbbm{1}_{d_{o_n} = k}d_{V_n} \le d_{o_n}^2$, the sequence $(\mathbbm{1}_{d_{o_n} = k}d_{V_n})_{n \ge 1}$ is uniformly integrable by our assumption. Hence, Lemma~\ref{lem:local_convergence_tool} implies that
\[
	\CExp{\mathbbm{1}_{d_{o_n} = k} d_{V_n}}{G_n} \plim \Expc{\mu}{\mathbbm{1}_{d_o} d_V}.
\]
Finally, we note that
\[
	\frac{\mathbbm{1}_{f_{G_n}(k) > 0}}{f_{G_n}(k)} \plim \mu(d_o = k)^{-1},
\]
which then implies that
\[
	\psi_{G_n}(k) \plim \frac{\Expc{\mu}{\mathbbm{1}_{d_o = k} d_V}}{\mu(d_o = k)} = \Expc{\mu}{d_V | d_o = k}.
\]
\end{proof}

The proof of the average nearest neighbor rank follows the same lines, after replacing $F^\ast_{G_{n}}$ with $F^\ast_\mu$. Here we only need uniform convergence of $d_{o_n}$ since we deal with $\mathbbm{1}_{d_{o_n} = k} F^\ast_\mu(d_{V_n})$ instead of $\mathbbm{1}_{d_{o_n} = k} d_{V_n}$. We include it here for completeness. 

\begin{proof}[Proof of Theorem~\ref{thm:annr}]

Following similar computations as in the previous proof, we write
\begin{align*}
	\theta_{G_n}(k) &= \frac{k \mathbbm{1}_{f_{G_n}(k) > 0}}{f^\ast_{G_{n}}(k)} \frac{1}{|\vec{E}(G_n)|}
		\sum_{u \in V(G_n)} \mathbbm{1}_{d_u = k} \sum_{v \in V(G_n), \, u \sim v} \frac{1}{d_u }F^\ast_{G_{n}}(d_v) 
		\numberthis \label{eq:annr_easy_expression} \\
	&= \frac{\mathbbm{1}_{f_{G_n}(k) > 0}}{f_{G_n}(k)} \CExp{\mathbbm{1}_{d_{o_n} = k} F^\ast_{G_n}(d_{V_n})}{G_n}.
\end{align*}

Let us now replace $F^\ast_{G_{n}}$ with $F^\ast_\mu$. Then since $\mathbbm{1}_{d_{o_n} = k} F^\ast_\mu(d_{V_n}) \le d_{o_n}$ and $(d_{o_n})_{n \ge 1}$ is uniformly integrable, Lemma~\ref{lem:local_convergence_tool} implies that
\[
	\CExp{\mathbbm{1}_{d_{o_n} = k} F^\ast_\mu(d_{V_n})}{G_n} \plim \Expc{\mu}{\mathbbm{1}_{d_o = k} F_\mu(d_V)}
\]
Together with,
\[
	\frac{\mathbbm{1}_{f_{G_n}(k) > 0}}{f_{G_n}(k)} \plim \mu(d_o = k)^{-1},
\]
we have
\[
	\frac{\mathbbm{1}_{f_{G_n}(k) > 0}}{f_{G_n}(k)} \CExp{\mathbbm{1}_{d_{o_n} = k} F^\ast_\mu(d_{V_n})}{G_n}
	\plim \frac{\Expc{\mu}{\mathbbm{1}_{d_o = k} F_\mu(d_V)}}{\mu(d_o = k)} = \Expc{\mu}{F_\mu(d_V) | d_o = k}.
\]

Thus, we are left to show that
\[
	\frac{\mathbbm{1}_{f_{G_n}(k) > 0}}{f_{G_n}(k)}\left| \CExp{\mathbbm{1}_{d_{o_n} = k} F^\ast_{G_n}(d_{V_n})}{G_n}
	- \CExp{\mathbbm{1}_{d_{o_n} = k} F^\ast_\mu(d_{V_n})}{G_n}\right| \plim 0.
\]
Going back to~\eqref{eq:annr_easy_expression} we get that
\begin{align*}
	&\hspace{-20pt}\frac{\mathbbm{1}_{f_{G_n}(k) > 0}}{f_{G_n}(k)}\left| \CExp{\mathbbm{1}_{d_{o_n} = k} F^\ast_{G_n}(d_{V_n})}{G_n}
		- \CExp{\mathbbm{1}_{d_{o_n} = k} F^\ast_\mu(d_{V_n})}{G_n}\right| \\
	&\le \frac{\mathbbm{1}_{f_{G_n}(k) > 0}}{f^\ast_{G_{n}}(k)} \frac{1}{|\vec{E}(G_n)|}
	\sum_{u \in V(G_n)} \mathbbm{1}_{d_u = k} 
		\left|\sum_{v \in V(G_n), \, u \sim v} F^\ast_{G_{n}}(d_v) - F_\mu^\ast(d_v)\right| \\
	&\le \frac{\mathbbm{1}_{f_{G_n}(k) > 0}}{f^\ast_{G_{n}}(k)} \sup_{\ell \ge 0} 
		\left|F^\ast_{G_{n}}(\ell) - F_\mu^\ast(\ell)\right|.
\end{align*} 
The first term converges to $(k \Expc{\mu}{\mathbbm{1}_{d_o = k} d_o}/\Expc{\mu}{d_o})^{-1}$ while Lemma~\ref{lem:basic_convergence_results} implies that
\[
	\sup_{\ell \ge 0} \left|F^\ast_{G_{n}}(\ell) - F_\mu^\ast(\ell)\right| \plim 0,
\]
yielding the required result.
\end{proof}

\section{Proofs of application results}\label{sec:application_proofs}

\subsection{Rank-1 inhomogeneous random graphs}

\begin{proof}[Proof of Proposition~\ref{prop:annd_IRG}]
We first observe that 
\[
	\Prob{d_o = k} = \Exp{\frac{W^k}{k!}e^{-W}},
\]
which is positive for all $k \ge 1$ since $W$ is non-negative. In particular, $\Prob{d_o \ge 1} > 0$ and hence Theorem~\ref{thm:annd} implies that
\[
	\phi_n(k) \plim \CExp{d_V}{d_o = k} = \Exp{d_V},
\]
where the last equality is because $d_o$ and $d_V$ are independent. Using that $d_V$ has the size-biased distribution we get that $\Exp{d_V} = \Exp{d_o^2}/\Exp{W}$. Finally we note that $\Exp{d_o^2} = \Exp{W} + \Exp{W^2}$, from which the result follows.
\end{proof}

\subsection{Random geometric graphs}

To prove our results for random geometric graphs, we need some intermediate results for the local limit graph $G_\infty(R,p)$. The first results concerns the uniform integrability of the degrees of the root.

\begin{lemma}\label{lem:uniform_integrable_RGGs}
Let $(\mathrm{RGG}_n(p,R))_{n \ge 1}$ be a sequence of random geometric graphs with connection radius $R > 0$ and connection probability $p$, and let $o_n$ be a randomly sampled vertex in $\mathrm{RGG}_n(p,R)$. Then the sequence $(d_{o_n}^3)_{n \ge 1}$ is uniformly integrable.
\end{lemma}

Next we study the degree distribution of $d_V$ conditioned on $d_o = k$ and $d(o,V) = r$. Denote by $B_o(\delta)$ the $d$-dimensional Euclidean ball of radius $\delta$ around $o$, and define $B_{V}(\delta)$ in a similar manner. Then there are two types of neighbors of $o$: those that are not neighbors of $V$, and those that are. The first type of neighbors are made up of all nodes in $B_o(R) \setminus B_V(R)$, while the second type are the nodes in $B_o(R) \cap B_V(R)$. A similar argument holds for the neighbors of $V$. Of course, there will also be a number of joint neighbors of $o$ and $V$. This will depend on the distance $r = d(o,V)$, see also Figure~\ref{fig:RGG_neighborhood_o_v}. 

\begin{figure}
\centering
\begin{tikzpicture}
	\tikzstyle{vertex}=[fill, circle, inner sep=0pt, minimum size=5pt]
		\tikzstyle{edge}=[color=black,line width=1pt]
		
		
		\coordinate (o) at (0,0);
		\path (o)+(45:1.5) coordinate (v);
		
		\fill[color=white] (v) circle (2) (o) circle (2);
		\fill[white,even odd rule] (v) circle (2) (o) circle (2);
		\fill[color=gray,even odd rule] (v) circle (2) (o) circle (2);
		
		\draw node[vertex] at (o) {};
		\draw node[vertex] at (v) {};
		
		\path (o)+(0,-0.3) node {$o$};
		\path (v)+(0,-0.3) node {$v$};
		
		\draw[thick] (o) circle  (2cm);
		\draw[thick] (v) circle  (2cm);
		
		\path (o)+(190:2) coordinate (R);
		\path (o)+(150:0.1) coordinate (o1);
		\path (v)+(105:0.1) coordinate (v1);
		
		\draw[thick,dashed] (o) -- (R);
		\path (o)+(170:1.1) node {$R$};
		
		\path (o)+(80:0.9) node {$r$};
		\draw[line width=1pt,decorate,decoration={brace,amplitude=5pt}] (o1) -- (v1);
\end{tikzpicture}
\caption{Joint neighborhood of the root $o$ and uniform neighbor $V$ at distance $r$ for a $2$-dimensional Random Geometric Graph with radius $R$. The green areas contain the first type of neighbors and the purple area contains the joint neighbors (second type).}
\label{fig:RGG_neighborhood_o_v}
\end{figure}
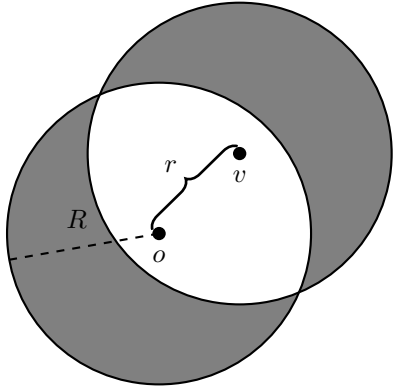
 
In the infinite Poisson limit $G_\infty(R,p)$, the nodes form a unit-rate Poisson point process on $\mathbb{R}^d$, and each pair of nodes within distance $R$ is connected independently with probability $p$. Under this model, the number of points in any region $A$ is Poisson with mean $\mathrm{vol}(A)$, and the processes in disjoint regions are independent. Furthermore, the event that a node is a neighbor of $o$ can be seen as an independent thinning of the Poisson process inside $B_o(R)$ with probability $p$. Consequently, the nodes that are \emph{not} neighbors of $o$ in any subset of $B_o(R)$ again form a Poisson process, now with intensity $(1-p)$, independent of the neighbors of $o$.
 
Now condition on the event that $d_o = k$ and that $V$ is a uniformly chosen neighbor of $o$ at distance $r$. One of the $k$ neighbors is $V$ itself, the other $k-1$ neighbors are located uniformly in $B_o(R)$. The number of these $k-1$ points that fall in the intersection $B_o(R) \cap B_V(R)$ is therefore
\[
Z_{k-1}(r) \sim \mathrm{Bin}\Bigl(k-1,\; \frac{\lambda_2(r)}{\omega_d(R)}\Bigr).
\]
We can now identify three independent sources of neighbors for $V$.

\begin{enumerate}
\item \textbf{The disjoint region $B_V(R) \setminus B_o(R)$.}  
  This region is outside $B_o(R)$, so its points are automatically not neighbors of $o$. Their number is $\mathrm{Po}(\lambda_1(r))$, and each connects to $V$ with probability $p$. By Poisson thinning, the number of neighbors of $V$ coming from this region is
  \[
  X(r) \sim \mathrm{Po}(p\lambda_1(r)),
  \]
  which is independent of everything happening inside $B_o(R)$.
\item \textbf{The other neighbors of $o$ that lie in the intersection.}  
  The $Z_{k-1}(r)$ points described above are already neighbors of $o$. For each of them, the connection to $V$ is an independent Bernoulli trial with success probability $p$. Hence they contribute
  \[
  \mathrm{Bin}(Z_{k-1}(r),\, p)
  \]
  additional neighbors to $V$.
\item \textbf{The points in the intersection that are \emph{not} neighbors of $o$.}  
  Inside $B_o(R) \cap B_V(R)$, the Poisson points that do \emph{not} connect to $o$ form a Poisson process with intensity $(1-p)$, by the independent thinning property. Their number is $\mathrm{Po}((1-p)\lambda_2(r))$ and this random variable is independent of the conditioning on $d_o = k$, since it concerns exactly the failures of the thinning that produced the $k$ neighbors. Each of these points connects to $V$ with probability $p$, yielding a $\mathrm{Po}(p(1-p)\lambda_2(r))$ number of extra neighbors for $V$. This contribution is independent of the previous two.
\end{enumerate}

Adding the edge between $o$ and $V$ itself, we obtain the exact conditional distribution
\begin{equation}\label{eq:conditional_degree_distribution_V_RGG}
	d_V \mid \{ d_o = k,\; d(o,V) = r \} \;\stackrel{d}{=}\; 1 \;+\; X(r) \;+\; \mathrm{Bin}\bigl(Z_{k-1}(r),\, p\bigr) \;+\; \mathrm{Po}\bigl(p(1-p)\lambda_2(r)\bigr).
\end{equation}

The distinction between the binomial and the Poisson terms in the intersection is crucial. The binomial term stems from the fixed number $k-1$ of other neighbors of $o$, while the extra Poisson term accounts for the (random) nodes in the intersection that are not neighbors of $o$ but can still link to $V$. Together with $X(r)$, this fully characterizes the degree of the uniform neighbor $V$ given $d_o$ and the distance $r$.

Now that we have the degree distribution of $V$ conditioned on its distance to $o$ and the degree $d_o$, we need to understand how the distance behaves.

\begin{lemma}[Distance to a uniform neighbor]\label{lem:distance_to_V}
Let $V$ be a uniform neighbor of $o$ in $G_\infty(R,p)$ and denote by $r_V := d(o,d_V)$ the distance between $o$ and $V$. Then $r_V$ has probability density function
\[
	f_{r_V}(r) = \frac{d}{R^d} r^{d-1} \mathbbm{1}_{0 \le r \le R}.
\]
\end{lemma} 

With this lemma, we can now compute the expected degree of a uniform neighbor of the root, conditioned on the degree of the root.

\begin{lemma}[Conditional expected degree of a uniform neighbor.]\label{lem:expected_degree_neighbor_RGG}
Let $V$ be a uniform neighbor of $o$ in $G_\infty(R,p)$. Then
\[
	\CExp{d_V}{d_o = k} = 1 + p\omega_d(R)(1 - p\,p_\mathrm{conn}) + p(k-1)p_\mathrm{conn}.
\]
\end{lemma}

We now have everything needed to prove the main results on Random Geometric Graphs. We start with Pearson's correlation coefficient (Proposition~\ref{prop:pearson_RGGs}) and then move to Average Nearest Neighbor Degree (Proposition~\ref{prop:annd_RGGs}).

\begin{proof}[Proof of Proposition~\ref{prop:pearson_RGGs}]
From Lemma~\ref{lem:expected_degree_neighbor_RGG} it follows that
\[
\CExp{d_o^2 d_V}{d_o = k} = k^2 \CExp{d_V}{d_o = k}
= k^2\bigl(1 + p\omega_d(R)(1 - p\,p_\mathrm{conn}) + p(k-1)p_\mathrm{conn}\bigr).
\]
Let $\mu = p\,\omega_d(R)$. Since $d_o \sim \mathrm{Po}(\mu)$,
\[
\mathbb{E}[d_o^2 d_V]
= \sum_{k=1}^{\infty} k^{2} \, \CExp{d_V}{d_o = k} \, \mathbb{P}(d_o = k).
\]
Inserting the expression for $\CExp{d_V}{d_o = k}$ and using $p\omega_d(R) = \mu$,
\begin{align*}
\mathbb{E}[d_o^2 d_V]
&= \sum_{k=1}^{\infty} e^{-\mu}\frac{\mu^{k}}{k!}
k^{2}\bigl[ 1 + \mu - p^{2}\omega_d(R)p_{\mathrm{conn}} + p(k-1)p_{\mathrm{conn}} \bigr] \\
&= \bigl(1 + \mu - p^{2}\omega_d(R)p_{\mathrm{conn}}\bigr) \mathbb{E}[d_o^{2}]
+ p\,p_{\mathrm{conn}}\,\mathbb{E}[d_o^{2}(d_o-1)].
\end{align*}
Using the Poisson moment identities
\[
\mathbb{E}[d_o^{2}] = \mu + \mu^{2}, \qquad
\mathbb{E}[d_o^{2}(d_o-1)] = \mu^{3} + 2\mu^{2},
\]
we obtain
\begin{align*}
\mathbb{E}[d_o^{2} d_V]
&= \bigl(1 + \mu - p^{2}\omega_d(R)p_{\mathrm{conn}}\bigr)(\mu + \mu^{2}) + p\,p_{\mathrm{conn}}(\mu^{3} + 2\mu^{2}) \\
&= (\mu + 2\mu^{2} + \mu^{3}) - p^{2}\omega_d(R)p_{\mathrm{conn}}(\mu + \mu^{2}) + p\,p_{\mathrm{conn}}(\mu^{3} + 2\mu^{2}).
\end{align*}
Now substitute $p^{2}\omega_d(R) = p\mu$ to simplify the middle term
\[
p^{2}\omega_d(R)p_{\mathrm{conn}}(\mu + \mu^{2}) = p\mu\,p_{\mathrm{conn}}(\mu + \mu^{2}) = p\,p_{\mathrm{conn}}(\mu^{2} + \mu^{3}).
\]
Hence
\begin{align*}
\mathbb{E}[d_o^{2} d_V]
&= \mu + 2\mu^{2} + \mu^{3} - p\,p_{\mathrm{conn}}(\mu^{2} + \mu^{3}) + p\,p_{\mathrm{conn}}(\mu^{3} + 2\mu^{2}) \\
&= \mu + 2\mu^{2} + \mu^{3} + p\,p_{\mathrm{conn}}\,\mu^{2}.
\end{align*}
We also have
\[
\frac{\mathbb{E}[d_o^{2}]^{2}}{\mathbb{E}[d_o]} = \frac{(\mu + \mu^{2})^{2}}{\mu} = \mu + 2\mu^{2} + \mu^{3}.
\]
Thus
\[
\mathbb{E}[d_o^{2} d_V] - \frac{\mathbb{E}[d_o^{2}]^{2}}{\mathbb{E}[d_o]} = p\,p_{\mathrm{conn}}\,\mu^{2}.
\]
Finally, the denominator for Pearson's correlation is
\[
\mathbb{E}[d_o^{3}] - \frac{\mathbb{E}[d_o^{2}]^{2}}{\mathbb{E}[d_o]} = \mu^{2},
\]
which follows from Poisson identities. Applying Theorem~\ref{thm:pearson} we conclude
\[
r(\mathrm{RGG}_n(p,R)) \xrightarrow{\mathbb{P}}  p\,p_{\mathrm{conn}}.
\]
\end{proof}

\begin{proof}[Proof of Proposition~\ref{prop:annd_RGGs}]
Recall that $d_o \stackrel{d}{=} \mathrm{Po}(p\omega_d(R))$. In particular, $\Prob{d_o \ge 1} > 0$ and by Lemma~\ref{lem:uniform_integrable_RGGs} the sequence $(d_{o_n}^2)_{n \ge 1}$ is uniformly integrable. Thus Theorem~\ref{thm:annd} implies that
\[
	\phi_n(k) \plim \CExp{d_V}{d_o = k}.
\]
The result then immediately follows from Lemma~\ref{lem:expected_degree_neighbor_RGG}.
\end{proof}

\paragraph{Acknowledgments}
The results in this paper were established during the bachelor project of the first author, under supervision of the second author. The first author is grateful for the invaluable guidance and insightful discussions throughout the project. We also want to thank Remco van der Hofstad for his constructive feedback and questions during the thesis presentation, especially on the conditions for ANNR, and to Vlad Mihai Ciuperceanu for his meticulous proofreading of the first version of the manuscript. The authors also extend their gratitude to two anonymous referees for important insights and comments on previous versions of the manuscript.

\appendix

\section{Proof of Lemma~\ref{lem:basic_convergence_results}}\label{sec:proof_appendix}
\begin{proof}
\hfil\\
\begin{enumerate}
\item Fix $\epsilon>0$ and let $K:= K(\epsilon)$ be such that $F_\mu^{*}(k) > 1 - \frac{\epsilon}{2}$ for all $k \ge K(\epsilon)$, which exists since $\limsup_{k \to \infty} F_\mu^{*}(k) = 1$. This then implies that for all $n \ge 1$
\begin{equation}\label{eq:sup_ge_K}
    \sup_{k > K(\epsilon)}\left\vert F^{*}_{G_{n}}(k)-F_\mu^{*}(k)\right\vert \le \max{\{\frac{\epsilon}{2}, \left\vert F^{*}_{G_{n}}(K(\epsilon)) -1\right\vert\}}.
\end{equation}
and hence
\begin{align*}
	\Prob{\sup_{k > K(\epsilon)}\left\vert F^{*}_{G_{n}}(k)-F_\mu^{*}(k)\right\vert > \epsilon}
	&\le \Prob{\left\vert F^{*}_{G_{n}}(K(\epsilon)) -1\right\vert >\epsilon}\\
	&\le \Prob{\left\vert F^{*}_{G_{n}}(K(\epsilon)) -F_\mu^{*}(K(\epsilon))\right\vert	
		+ \frac{\epsilon}{2} > \epsilon} \\
	&= \Prob{\left\vert F^{*}_{G_{n}}(K(\epsilon)) -F_\mu^{*}(K(\epsilon))\right\vert	> \epsilon/2}.
\end{align*}
Since $K(\epsilon)$ is fixed, the last term converges to zero in probability as $n \to \infty$.

For the other part of the supremum we have
\begin{align*}
	\Prob{\sup_{k \le K(\epsilon)}\left\vert F^{*}_{G_{n}}(k)-F_\mu^{*}(k)\right\vert > \epsilon}
	\le \sum_{k=1}^{K(\epsilon)}\Prob{\left\vert F^{*}_{G_{n}}(k)-F_\mu^{*}(k)\right\vert > \epsilon},
\end{align*}
which converges to zero in probability since each of the finite term in the sum does.

Together we conclude that for any $\epsilon > 0$
\[
	\lim_{n \to \infty} \Prob{\sup_{k \ge 0}\left\vert F^{*}_{G_{n}}(k)-F_{\mu}^{*}(k)\right\vert > \epsilon} = 0,
\]
as required.
\item 
Following a similar approach as for the proof of the previous part, we fix $\varepsilon > 0$ and pick $M:= M(\epsilon)$ be such that $F_\mu(m) > 1 - \frac{\epsilon}{2}$ for all $m \ge M(\epsilon)$. Next, since $H_\mu(k,\ell) \to F_\mu(k)$ for $\ell \to \infty$, we can pick for each $k$ an $L_k$ such that $H_\mu(k,\ell) > 1 - \epsilon/2$ for all $\ell \ge L_k$. In a similar fashion we pick $K_\ell$ such that $H_\mu(k,\ell) > 1 - \epsilon/2$ for all $k \ge K_\ell$. Now we set $L := \max_{k \le M} L_k$ and $K := \max_{\ell \le M} K_\ell$. 

We then get that for all $\ell \ge 0$
\[
	\sup_{k > K}\left\vert H_{G_{n}}(k,\ell)-H_\mu(k,\ell)\right\vert \le \max{\{\frac{\epsilon}{2}, \left\vert 
	H_{G_{n}}(K,\ell) -1\right\vert\}},
\]
and similarly, for any $k \ge 0$
\[
	\sup_{\ell > L}\left\vert H_{G_{n}}(k,\ell)-H_\mu(k,\ell)\right\vert \le \max{\{\frac{\epsilon}{2}, \left\vert 
	H_{G_{n}}(k,L) -1\right\vert\}}.
\]

Then, using similar consideration as in the previous proof, we get that
\[
	\Prob{\sup_{k > K, \ell > L}\left\vert H_{G_{n}}(k,\ell)-H_\mu(k,\ell)\right\vert > \epsilon} \plim 0.
\]

We are now left with three remaining terms:
\begin{align}
	&\sup_{k \le K, \ell \le L} \left\vert H_{G_{n}}(k,\ell)-H_\mu(k,\ell)\right\vert, \label{eq:H_part1}\\
	&\sup_{k \le K, \ell > L} \left\vert H_{G_{n}}(k,\ell)-H_\mu(k,\ell)\right\vert, \label{eq:H_part2}\\
	&\sup_{k > K, \ell \le L} \left\vert H_{G_{n}}(k,\ell)-H_\mu(k,\ell)\right\vert. \label{eq:H_part3}
\end{align}
The first term converges to zero in probability since $K$ and $L$ are fixed. The second term converges to zero since
\begin{align*}
	\Prob{\sup_{k \le K, \ell > L} \left\vert H_{G_{n}}(k,\ell)-H_\mu(k,\ell)\right\vert > \varepsilon}
	&\le \Prob{\sup_{k \le K} \left|H_{G_{n}}(k,L)-H_\mu(k,L)\right| > \varepsilon/2} \\
	&\le \sum_{k \le K} \Prob{\left|H_{G_{n}}(k,L)-H_\mu(k,L)\right| > \varepsilon/2},
\end{align*}
and $K$ is fixed. The third term converges to zero by a similar reasoning.
\item 
We first bound the difference between the two terms as follows
\begin{align*}
	\left| \CExp{F_{G_n}^\ast(d_{o_n)}}{G_n} - \Expc{\mu}{F_\mu^\ast(d_o)}\right|
	&\le \left|\CExp{F_{G_n}^\ast(d_{o_n)}}{G_n} - \CExp{F_\mu^\ast(d_{o_n)}}{G_n}\right| \\
	&+\hspace{10pt} \left|\CExp{F_\mu^\ast(d_{o_n)}}{G_n} - \Expc{\mu}{F_\mu^\ast(d_o)}\right|.
\end{align*}

The first term is bounded by $\sup_{k \ge 0} \left|F^\ast_{G_n}(k) - F_\mu^\ast(k)\right|$ and converges to zero by 1. For the second term we note that the map $(G,o) \mapsto F_\mu(d_o)$ is bounded and continuous. Hence, by Lemma~\ref{lem:local_convergence_tool} this terms converges to zero in probability as well.

\item Note that 
\[
	\frac{|\vec{E}(G_n)|}{|V(G_n)|} = \frac{2}{|V(G_n)|} \sum_{v \in V(G_n)} d_v = 2\CExp{d_{o_n}}{G_n}.
\]
Since $(d_{o_n})_{n \ge 1}$ is uniformly integrable Lemma~\ref{lem:local_convergence_tool} implies that $|\vec{E}(G_n)|/|V(G_n)| \plim 2\Expc{\mu}{d_o}$. Finally, $\Prob{d_o \ge 1} > 0$ then implies that $2|V(G_n)|/|\vec{E}(G_n)| \plim 1/\Expc{\mu}{d_o}$.
\end{enumerate}

\end{proof}
\section{Proof of Lemma~\ref{lem:uniform_integrable_RGGs}}
\begin{proof}
    Fix $\epsilon >0$. We need to show that there exist $M_0, N \in \mathbb{N} $ such that for all
$M > M_0$ and $n > N = N(M_0)$,
\begin{align*}
    \mathbb{E}[d^{3}_{o_{n}}\mathbf{1}_{\{d_{o_{n}} > M\}}] < \epsilon.
\end{align*}
Consider the random geometric graph $\mathrm{RGG}_n(p,R)$ defined on the $d$-dimensional torus $\mathcal{T}_n = [-n^{1/d}/2, n^{1/d}/2]^d$. Let $n$ large enough so that the ball $B(0, R)$ is contained in $\mathcal{T}_n$. The graph $G_{n}$ has $n$ vertices $X_1, \dots, X_n$ positioned independently and uniformly at random in $\mathcal{T}_n$. Fix a vertex $o_n$ chosen uniformly at random from $\{X_1, \dots, X_n\}$, and condition on its location $X_{o_n} = x$. By translation invariance we can assume w.l.o.g. that $x=O$ so that $X_{o_n}$ is at the origin. The remaining $n-1$ vertices are then independent and uniformly distributed over $\mathbb{T}_n^d$. For each such vertex with position $X_j$, define the indicator
\[
Y_j := \mathbbm{1}_{\{\|X_j\|_n \le R\}},
\]
and note that, conditional on $X_{o_n} = O$, the degree of the root can be expressed as $d_{o_n} = \sum_{j} Y_j $.
Since
\[
\mathbb{P}(X_j \in B(x,R) \mid X_{o_n} = O) = \frac{\mathrm{Vol}(B(x,R) \cap \mathbb{T}_n^d))}{\mathrm{Vol}(\mathbb{T}_n^d)} = \frac{\omega_d(R)}{n},
\]
it follows that, conditionally on $X_{0_n} = O$, $d_{o_n}$ is stochastically dominated by the binomial distribution 
\[ 
	Z_{n} :\stackrel{d}{=} \mathrm{Bin}\Big(n-1, \frac{p \omega_d(R)}{n}\Big).
\]

Since the edge exists independently with probability $p$, we obtain
\[
\mathbb{P}\big(Y_j = 1 \mid X_{o_n} = x\big) \le  p \cdot \frac{\omega_d(R)}{n}.
\]

Because $R$  is fixed, the success probability of $Z_{n}$ is of order $1/n$. Thus, for $n$ large enough, the fourth moment of $Z_n$ is uniformly bounded
\[
\sup_{n \ge n_{0}} \mathbb{E}[Z_n^4] < \infty,~ \text{for some $n_{0}$},
\] 
which implies by Markov's inequality that for any $M>0$,
\[
\mathbb{E}[Z_n^3 \mathbf{1}_{\{Z_n > M\}}] \le \frac{\mathbb{E}[Z_n^4]}{M} \le \frac{C}{M},
\]
where $C := \sup_{n \ge n_{0}} \mathbb{E}[Z_n^4]$. 

Picking $M$ large enough, we thus conclude that 
\[
\mathbb{E}[d_{o_{n}}^3 \mathbf{1}_{\{d_{o_{n}} > M\}}] \le \mathbb{E}[Z_n^3 \mathbf{1}_{\{Z_n > M\}}] < \epsilon,
\]
and hence that $(d_{o_n}^3)_{n \ge 1}$ is uniformly integrable.
\end{proof}

\section{Proof of Lemma~\ref{lem:distance_to_V}}
\begin{proof}
Recall that $d_0$ is the degree of the root and that $d_0 \stackrel{d}{=} \mathrm{Po}(p \omega_d(R))$, where $\omega_d(R) = v_d R^d$ is the volume of the $d$-dimensional ball or radius $R$ and $v_d$ is the volume of the unit ball. 

Let $r_V$ denote the distance from the root $o$ to a uniformly chosen neighbor. For $s \in [0, R]$, define the cumulative distribution function (CDF) of $r_V$ conditioned on $d_0 = k$ with $k \ge 1$ by
\[
F_{r_V}(s | k) := \mathbb{P}(r_V \le s \mid d_0 = k),
\]
Consider the infinite random geometric graph $G_\infty(R,p) := \Phi \cup \{o\}$, where $\Phi$ is a homogeneous unit-rate Poisson point process on $\mathbb{R}^d$. Let $N_s$ denote the number of neighbors of the root within the ball $B(o,s)$ and $M_s$ the number of neighbors in $B(o,R) \setminus B(o,s)$. In particular, 
\[
	N_s \stackrel{d}{=} \mathrm{Po}(p v_d s^d)
\]
and thus, conditioned on $d_o = k$ we have that 
\[
N_s | d_o = k \, \stackrel{d}{=} \mathrm{Bin}(k, \frac{p v_d s^d}{p\omega_d(R)}) \stackrel{d}{=} \mathrm{Bin}(k, s^d/R^d).
\]
Using the law of total expectation, we then obtain
\begin{align*}
\mathbb{P}(r_V \le s \mid d_0 = k) 
&= \mathbb{E}\big[ \mathbf{1}_{\{r_V \le s\}} \mid d_0 = k  \big] \\
&= \mathbb{E}\big[ \mathbb{E}[ \mathbf{1}_{\{r_V \le s\}} \mid \Phi ] \mid d_0 = k \big] \\
&= \mathbb{E}\Big[ \frac{N_s}{d_0} \mid d_0 = k \Big] \\
&= k^{-1} \mathbb{E}\Big[ N_s \mid d_o = k \Big] \\
&= \frac{s^d}{R^d}
\end{align*}

We thus conclude that
\[
F_{r_V}(s|k) = \mathbb{E}\Big[ \frac{N_s}{N_R} \mid N_R \ge 1 \Big] = \frac{s^d}{R^d}, \quad 0 \le s \le R,
\]
and we set $F_{r_V}(s|k) = 0$ for $s < 0$ and $F_{r_V}(s|k) = 1$ for $s > R$. Differentiating, the probability density function is
\[
f_{r_V}(r|k) = \frac{d}{dr} F_{r_V}(r|k) = \frac{d}{R^d} r^{d-1}, \quad 0 \le r \le R.
\]
\end{proof}

\section{Proof of Lemma~\ref{lem:expected_degree_neighbor_RGG}}

\begin{proof}
Recall that we need to compute $\CExp{d_V}{d_o = k}$. For this, we first condition on the distance $r_V := d(d_o, d_V)$. It then follow from~\eqref{eq:conditional_degree_distribution_V_RGG} that
\begin{align*}
	\CExp{d_V}{d_o = k, d(o,V) = r} 
	& = 1 + p\lambda_1(r) + (k-1)\frac{\lambda_2(r)}{\omega_d(R)} \\
	&= 1 + p(\omega_d(R) - \lambda_2(r)) + (k-1)\frac{\lambda_2(r)}{\omega_d(R)}.
\end{align*}
We thus have to compute the expectation of $\lambda_2(r)$ is taken with respect to the probability density function of $r_V := d(o,V)$ from Lemma~\ref{lem:distance_to_V}. This yields
\begin{align*}
	\mathbb{E}_{r_V}[\lambda_2(r_V)] &= \int_0^R \lambda_2(r) d R^{-d} r^{d-1} \, d r\\
	&= \frac{1}{\omega_d(R)} \int_0^R \lambda_2(r) d \omega_d r^{d-1} \, d r\\
	&=  \frac{1}{\omega_d(R)} \int_{x \in B_0(R)} \lambda_2(\|x\|) \, d x \\
	&= \frac{1}{\omega_d(R)} \iint_{x,y \in B_0(R)} \mathbbm{1}_{\|x-y\| \le R} \, dy \, dx \\
	&= \omega_d(R) p_\mathrm{conn}.
\end{align*}
We now conclude that
\begin{equation}\label{eq:conditional_expected_degree_V}
	\CExp{d_V}{d_o = k} = 1 + p\omega_d(R) - p \omega_d(R) p_\mathrm{conn} + (k-1)p_\mathrm{conn}
	= 1 + p\omega_d(R)(1 - p_\mathrm{conn}) + (k-1)p_\mathrm{conn},
\end{equation}
which finishes the proof.
\end{proof}

\end{document}